\documentclass[10pt,reqno]{amsart}

\usepackage[cp1251]{inputenc}

\usepackage{amssymb
,epsfig
}
\usepackage{graphicx}

\newtheorem{dfn}{Definition}
\newtheorem{thm}{Theorem}
\newtheorem{prp}{Proposition}
\newtheorem{lemma}{Lemma}

\newtheorem{ex}{Example}
\newtheorem{rem}{Remark}
\begin{document}

\title{Critical configurations of planar robot arms}

\author{G.Khimshiashvili*, G.Panina$^\dag$, D.Siersma$^\ddag$, A.Zhukova$^\checkmark$}

\address{*Ilia Chavchavadze State University, Tbilisi, Georgia,
e-mail: giorgi.khimshiashvili@iliauni.edu.ge. $^\dag$ Institute for
Informatics and Automation, St. Petersburg, Russia, Saint-Petersburg
State University, St. Petersburg, Russia,
e-mail:gaiane-panina@rambler.ru. $^\ddag$ University of Utrecht,
Utrecht, The Netherlands, e-mail: D.Siersma@uu.nl. $^\checkmark$
Saint-Petersburg State University, St. Petersburg, Russia, e-mail:
millionnaya13@ya.ru }


\begin{abstract} It is known that a closed polygon $P$ is a critical
point of the oriented area function if and only if $P$ is a cyclic
polygon, that is, $P$ can be inscribed in a circle. Moreover, there
is a short formula for the Morse index. Going further in this
direction, we extend these results to the case of open polygonal
chains, or robot arms. We introduce the notion of the oriented area
for an open polygonal chain, prove that critical points are exactly
the cyclic configurations with antipodal endpoints and derive a
formula for the Morse index of a critical configuration.
\end{abstract}

\maketitle

\section{Introduction}

Geometry of various special configurations of robot arms modeled by
open polygonal chains appears essential in many problems of
mechanics, robot engineering and control theory. The present paper
is concerned with certain planar configurations of robot arms
appearing as critical points of the oriented area considered as a
function on the moduli space of the arm in question. This setting
naturally arose in the framework of a general approach to extremal
problems on configuration spaces of mechanical linkages developed in
\cite{khi2008}, \cite{khipan}, \cite{khisie}, which has led to a
number of new results on the geometry of cyclic polygons
\cite{panzh}, \cite{kpsz} and suggested a variety of open problems.
The approach and results of \cite{khi2008}, \cite{khipan} provided a
paradigm and basis for the developments presented in this paper.

Let us now outline the structure and main results of the paper. We
begin with recalling necessary definitions and basic results
concerned with moduli spaces and cyclic configurations. In the
second section we prove that critical configurations of a planar
robot arm are given by the cyclic configurations with diametrical
endpoints called diacyclic (Theorem 1) and describe the structure of
all cyclic configurations of a robot arm (Theorem 2). Next, we
establish that, for a generic collection of lengths of the links,
the oriented area is a Morse function on the moduli space (Theorem
3) and provide some explications in the case of a 3-arm. In the last
section we prove an explicit formula for the Morse index of a
diacyclic configuration (Theorem 6) and illustrate it by a few
visual examples. In conclusion we briefly discuss several open
problems and related topics.

\textbf{Acknowledgements.} We are grateful to ICTP, MFO, and CIRM.
It's our special pleasure to acknowledge the excellent working
conditions in these institutes.

\section{Oriented area function for planar robot arm}

Let $L = (l_1,\ldots,l_n), \  L \in \mathbb{R}^n_+$. Informally, a
\emph{ robot arm}, or an \emph{open polygonal chain} is defined as a
linkage built up from rigid bars (edges) of lengths $l_i$
consecutively joined at the vertices by revolving joints. It lies in
the plane, its vertices may move, and the edges may freely rotate
around endpoints and intersect each other. This makes various planar
configurations of the robot arm.

Let us make this precise. A \emph{configuration} of a robot arm is
defined as a $n+1$-tuple of points $R=(r_0,\ldots,r_n)$ in the
Euclidean plane $\mathbb{R}^2$ such that $|r_{i-1}r_i|=l_i, i = 1,
\ldots , n$. Each configuration carries a natural orientation given
by vertices' order.

 To factor
out the action of orientation-preserving isometries of the plane
$\mathbb{R}^2$, we consider only configurations with two first
vertices fixed: $r_0=(0,0)$, $r_1=(l_1,0)$. The set of all such
planar configurations of a robot arm is called the \emph{moduli
space} of a robot arm. We denote it by
 $ M^0(L)$. It is a subset of Euclidean
space $\mathbb{R}^{2n-2}$ and inherits its topology and a
differentiable space structure so that one can speak of smooth
mappings and diffeomorphisms in this context. After these
preparations it is obvious that the moduli space of any planar robot
arm is diffeomorphic to the torus $(S^1)^{n-1}$. We will use its
parametrization by angle-coordinates $\beta_i$ (that is, by angles
between $r_0r_1$ and $r_kr_{k+1}, k=1,\ldots,n-1$).

In this paper we consider the oriented (signed) area as a function
on $M^0(L)$.
\begin{dfn}
For any configuration $R$ of $L$ with vertices $r_i = (x_i, y_i), \
i = 0, \ldots , n$, its \emph{(doubled) oriented area }A(R) is
defined by
$$2A(R) = (x_0y_1 - x_1y_0) + \cdots+ (x_ny_0 - x_0y_n).$$
\end{dfn}

In other words, we add the connecting side $r_nr_0$ turning a given
configuration $R$ into a $(n+1)$-gon and compute the oriented area
of the latter. Obviously, $A(R)$ is a smooth function on  the moduli
space $M^0(L)$ of any robot arm $L$.

\section{Critical configurations. 3-arms. }

A configuration $R=(r_0,\ldots,r_n)$ of a robot arm $L =
(l_1,\ldots,l_n)$ is \emph{cyclic} if all its vertices lie on a
circle.

A configuration is \emph{quasicyclic } (a QC-configuration for
short) if all its vertices lie either on a circle or on a (straight)
line.

A configuration is \emph{closed cyclic} if the last and the first
vertices coincide: $r_0=r_n$.

A configuration is \emph{diacyclic} if it is cyclic and the
"connecting side" $r_n,r_0$ is a diameter of the circumscribed
circle ("diacyclic" is a sort of shorthand for "diametrally
cyclic"). In other words, the connecting side $r_nr_0$ passes
through the center of the circumscribed circle or, equivalently,
each interval $r_0r_k$ is orthogonal to the interval $r_kr_n$ for
$k=1,\ldots,n-1$.

\begin{thm} \label{Thm_crirical}
For any robot arm $L \in \mathbb{R}^n_+$, critical points of $A$ on
the moduli space $M^0(L)$ are exactly the diacyclic configurations
of $L$.
\end{thm}

\begin{proof} As above, we assume that $r_0 = (0, 0)$, \ $r_1 = (l_1, 0)$.
 For a configuration $R = (r_0,\ldots, r_n)$ we put $e_i = {r_i-r_{i+1}}, i = 1,\ldots, n$.
  Obviously, ${r_i} = {e_1} +\cdots+{e_i}$ and ${e_i} =
l_i(cos \beta_i, sin \beta_i)$. Denote by  ${a}\times {b}$ the
oriented area of the parallelogram spanned by vectors $a$ and ${b}$
(i.e., we take the third coordinate of their vector product). The
differentiation of vectors ${e_i}$ with respect to angular
coordinates $\beta_j$ will be denoted by upper dots (i.e., there
will appear terms of the form $\dot{e_i}$).
 With these assumptions and
notations we can write $$A = \sum^n_{j=1} r_{j-1} \times r_j =
\sum^n_{j=2}(e_1 +\cdots +e_{j-1}) \times e_j = \sum_ {1\leq i<j\leq
n} e_i \times e_j.$$
 Taking partial
derivatives with respect to $\beta_k,\  k = 2,\ldots, n$ we get
$$\frac{\partial A}{\partial \beta_k} = -\sum^{k-1}_{i=1} {e_i}\times
\dot{e_k} + \sum^{k-1}_{i=1} e_k\times \dot{e_i}.$$ Notice now the
identities: $$\dot{e_i }\times e_j = e_i \cdot e_j = -e_i
\times\dot{ e_j}.$$ Eventually we get: $$\frac{\partial A}{\partial
\beta_k} =-\sum^{k-1}_{i=1} e_k \cdot e_i + \sum^n_{i=k+1} e_k \cdot
e_i = (-\sum^{k-1}_ {i=1} e_i + \sum^n_{i=k+1} e_i)\cdot e_k.$$
Consider now the equations $\frac{\partial A}{\partial\beta_k} = 0,\
k = 2,\ldots, n$ defining the critical set of A. By taking
appropriate linear combinations of equations, this system of $n-1$
equations is easily seen to be equivalent to the system of
equations:
$$( \sum^{k-1}_{i=1} e_i)
\cdot ( \sum^n_{i=k} e_j) = 0, \  k = 2,\ldots,n.$$ In geometric
terms this means that the intervals $r_0r_{k-1}$ and $r_{k-1}r_n$
are orthogonal for $k = 2,\ldots, n$. It remains to refer to Thales
theorem  to conclude that the points $r_0,\ldots, r_n$ lie on a
circle with diameter $r_0r_n$.
\end{proof}

\begin{lemma}The order of the lengths $l_1,\dots,l_n$ does not
matter: for any transposition $\sigma$, there exists a
diffeomorphism taking $M^0(L)$ to $M^0(\sigma L)$  which preserves
the function $A$, and therefore, all the critical points together
with their Morse indices.
\end{lemma}
The proof (which repeats  the proof of the similar lemma for closed
polygons from \cite{panzh}) is as follows. Two consecutive edges of
a configuration can be (geometrically) permuted in such a way that
the oriented area remains unchanged. Such  a geometrical permutation
yields a diffeomorphism from one configuration space to another.\qed

\begin{thm}\label{Thm_Qasicyclic} Assume that $l_1>l_i$ for all $i=2,\dots,n$.
Then we have the following:
\begin{enumerate}
    \item The set of all quasicyclic
configurations is a disjoint collection of $2^{(n-2)}$ embedded
(topological) circles  (QC-components for short).
    \item Each of the circles contains at least two critical points of
$A$.
    \item Assuming that all critical points are Morse non-degenerate, $A$  is a perfect Morse function if and
    only if each circle has exactly two critical points of $A$.
    \item Each of the circles contain exactly two aligned
    configurations.
\end{enumerate}

\end{thm}
Proof. We shall use the following notation: For a quasicyclic
configuration, we define $\varepsilon_i =1$ if the center of the
circle lies to the left with respect to $r_{i-1}r_i$. Otherwise we
put $\varepsilon_i =-1$.

  We show that  each
collection of signs $\varepsilon_i=\pm 1, \
 i=3,\dots,n$ yields a (topological) circle of quasicyclic configurations.

 Indeed, fix $\varepsilon_3, \dots, \varepsilon_n$.
 Take a  (metric) circle $S(\rho)$ whose radius $\rho$ varies from $l_1$ to
 infinity.

A differentiable coordinate for a QC-component is e.g. the angle
between the first and the second arm (mod $2 \pi$). The change of
this angle induces a differentiable change of the radius $\rho$ and
each vertex moves around the intersection of a circle with center
$r_i$ and radius $l_i$, which intersects the circumscribed circle
(with radius $\rho$) transversal (due to the condition $l_1 > l_i$).


 If $l_1<\rho< \infty$, the circle $S(\rho)$ has exactly one (up to a rigid motion) inscribed configuration
 with $E=(\pm 1,1, \varepsilon_3,\varepsilon_4,
 \dots,\varepsilon_n)$ and
 exactly one  inscribed configuration
 with $E(R)=(\pm 1,-1, -\varepsilon_3,-\varepsilon_4, \dots,-\varepsilon_n)$.
 The QC-component becomes in this way divided into four arcs, each
 with prescribed
 type of $E$, parameterized by the radius $\rho$. At the endpoints
 (that is, if $\rho=l_1$ or $\rho=\infty$) the four arcs join.
  More precisely, the
 arc that corresponds to $1,1,\varepsilon_3, \dots,
 \varepsilon_n$ is followed by the arc  that corresponds to $-1,1,\varepsilon_3, \dots,
 \varepsilon_n$, then the next one with $+1,-1,-\varepsilon_3, \dots,
 -\varepsilon_n$, then to the one with $-1,-1,-\varepsilon_3, \dots,
 -\varepsilon_n$, and then to $1,1,\varepsilon_3, \dots,
 \varepsilon_n$  (see Fig. \ref{quasicyclic}).
 Continuity reasons imply that each such a circle of quasicyclic
 configurations has at least two diacyclic ones.

\begin{rem}
The condition $l_1>l_i$ is important indeed: if there are several
longest edges, the QC-components acquire common points. For
instance, for an equilateral arm, they form a connected set.
\end{rem}

\begin{rem}
A QC-component  can contain besides the diacyclic and aligned arms
also closed cyclic arms (polygons). All of them occur in this way.
These special configurations are related to critical points of
functions on configuration spaces (respectively, oriented area of an
arm, squared length of the closing interval (see \cite{MilKap}, and
oriented area of a polygon (see \cite{khipan}). Note that existence
of a closed polygon on a QC-component  (as well as the number of
diacyclic configurations) depends on $l_1,\dots,l_n$.
\end{rem}

\begin{figure}[h]
\centering
\includegraphics[width=10 cm]{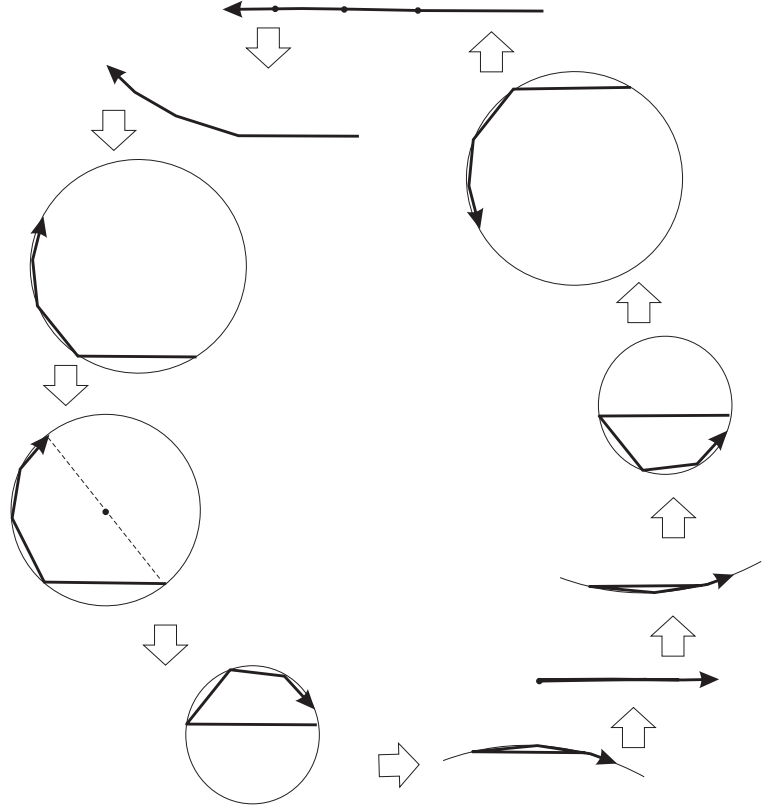}
\caption{A circle of quasicyclic configurations }\label{quasicyclic}
\end{figure}
 \qed
\begin{thm} For a generic sidelength vector $L \in\mathbb{ R}^n_+$, the function $A$ has
only non-degenerate critical points on $M^0(L)$.
\end{thm}
\begin{proof}
The proof from \cite{panzh} is applicable with some evident
modifications. Namely, after introducing a local coordinate system
with diagonals as coordinates, the Hessian matrix becomes
tridiagonal with analytic entries. Deformation arguments show that a
perturbation of just two of the edgelengths $l_i$  makes the Hessian
non-zero.
\end{proof}

For a $2$-arm we obviously have two points: one maximum and one
minimum.
\begin{prp}

Generically, for a 3-arm $A$ has exactly $4$ critical points on
$M^0(P)$. If $A$ is a Morse function (that is, if the Hessian is
non-degenerate), these are two extrema and two saddles (see Fig.
\ref{Fig_openarm}). Extrema are given by the convex diacyclic
configurations.
\end{prp}

\begin{figure}[h]
\centering
\includegraphics[width=14 cm]{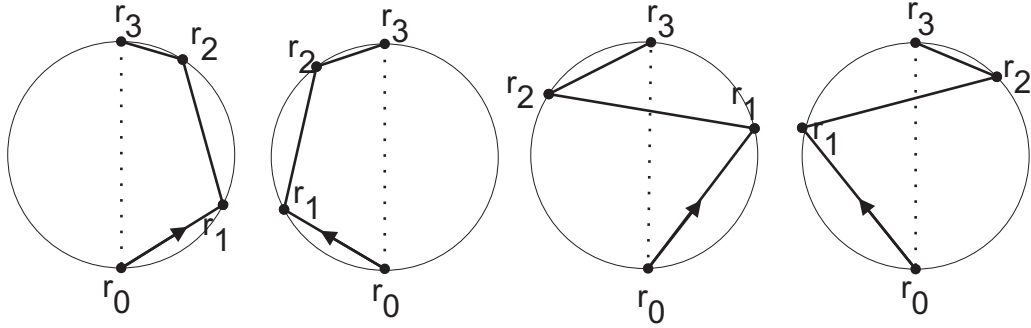}
\caption{Diacyclic configurations for a generic $3$-arm
}\label{Fig_openarm}
\end{figure}
{\bf Proof.} Partial derivatives give the conditions for critical
point:

$$ \partial A / \partial \beta_1 = l_1l_2 cos[\beta_1] - l_2l_3 cos[\beta_2-\beta_1] = 0 ,\,
\partial A / \partial \beta_2 = l_1l_3 cos[\beta_2]  + l_2l_3 cos[\beta_2-\beta_1]= 0 .$$
The orthogonality conditions are simply $ {r_0r_1}\bot ({r_1r_2}+
{r_2r_3}), \, ({r_0r_1}+ {r_1r_2}) \bot {r_2r_3}.$

The next step is to show that there are exactly $4$ critical points.
This can be done as follows. One uses elementary  geometry to obtain
a cubic equation for the length $d$ of the connecting edge: $$ d^3 =
(l_1^2+l_2^2+l_3^2)d \pm 2l_1l_2l_3 .$$

One has to solve these equations in $d$, taking into account $d \ge
l_1$ and $d \ge l_3$. Elementary calculation shows that both the +
equation and the - equation have one solution satisfying these
conditions. From $\cos \beta_1 = \pm l_3/d, \, \cos(\beta_2 -
\beta_1) = \pm l_1/d$ it follows that there are exactly two
solutions in each case. They occur in pairs $(\beta_1,\beta_2), \,
(-\beta_1,-\beta_2)$, which gives the result.

\bigskip

Notice that this reasoning shows in all cases, (except for
$l_1=l_2=l_3$) that there are four critical points; in the generic
case they are all Morse. If $l_1=l_2=l_3$  there are three critical
points, one of which is a "monkey saddle" . Note that in this case
we obtain the minimal number of critical points of a differentiable
function on a torus. It is equal to the Lusternick-Schnirelmann
category of the torus, see \cite{Takens}.

\begin{ex} In the case $l_1=l_2=l_3$, there are three critical
points on the torus: one maximum of $A$, one minimum, and a monkey
saddle point. Figure \ref{level}, left depicts the level sets of $A$
on the torus, whereas generically we have Figure \ref{level}, right.
\end{ex}

\begin{figure}[h]
\centering
\includegraphics[width=10 cm]{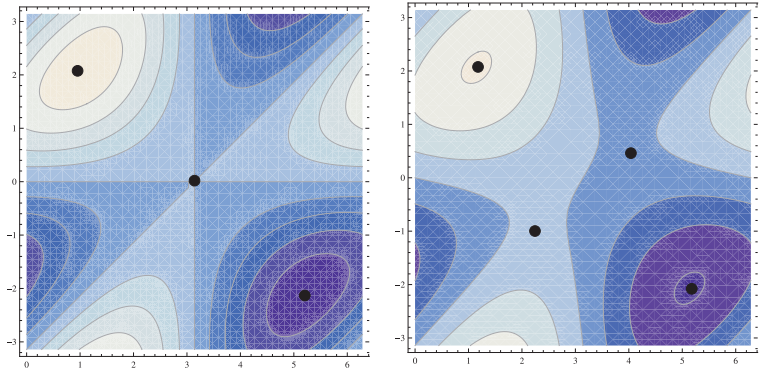}
\caption{Level sets of the function $A$ for $l_1=l_2=l_
3$}\label{level}
\end{figure}

\section{On Morse index  of a diacyclic configuration}

We start with some examples.

For arbitrary $n>3$, oriented area $A$ may or may  not be a perfect
Morse function:

\begin{ex}
Let $n=4$ and $L=(10,3,2,1)$. To be more precise, we take the
lengths generically perturbed in order to guarantee non-degenerate
critical points.  Then configuration space is  $M^0(L)=(S^1)^3$. Its
Betti numbers are $\beta_0=1$, $\beta_1=3$, $\beta_2=3$,
$\beta_3=1$. Direct computations show, that there are exactly 8
critical points on $M^0(L)$ (the four configurations depicted in
Fig. \ref{Fig_openarm} and their symmetric images). Therefore for
this particular linkage $A$ is a perfect Morse function.
\end{ex}

\begin{ex}
Let now $L=(22,17,21.9,19)$.

 Again, $M^0(L)=(S^1)^3$. However,
direct computations show, that there are more than 8 critical points
on $M^0(L)$ (the six configurations depicted in Fig.
\ref{Fig_openarmnonmorse} and their symmetric images). Therefore in
this case $A$ is not a perfect Morse function. There are two
QC-components with 3 diacyclic configurations, whereas all others
have only one.

\end{ex}
\begin{figure}[h]
\centering
\includegraphics[width=6 cm]{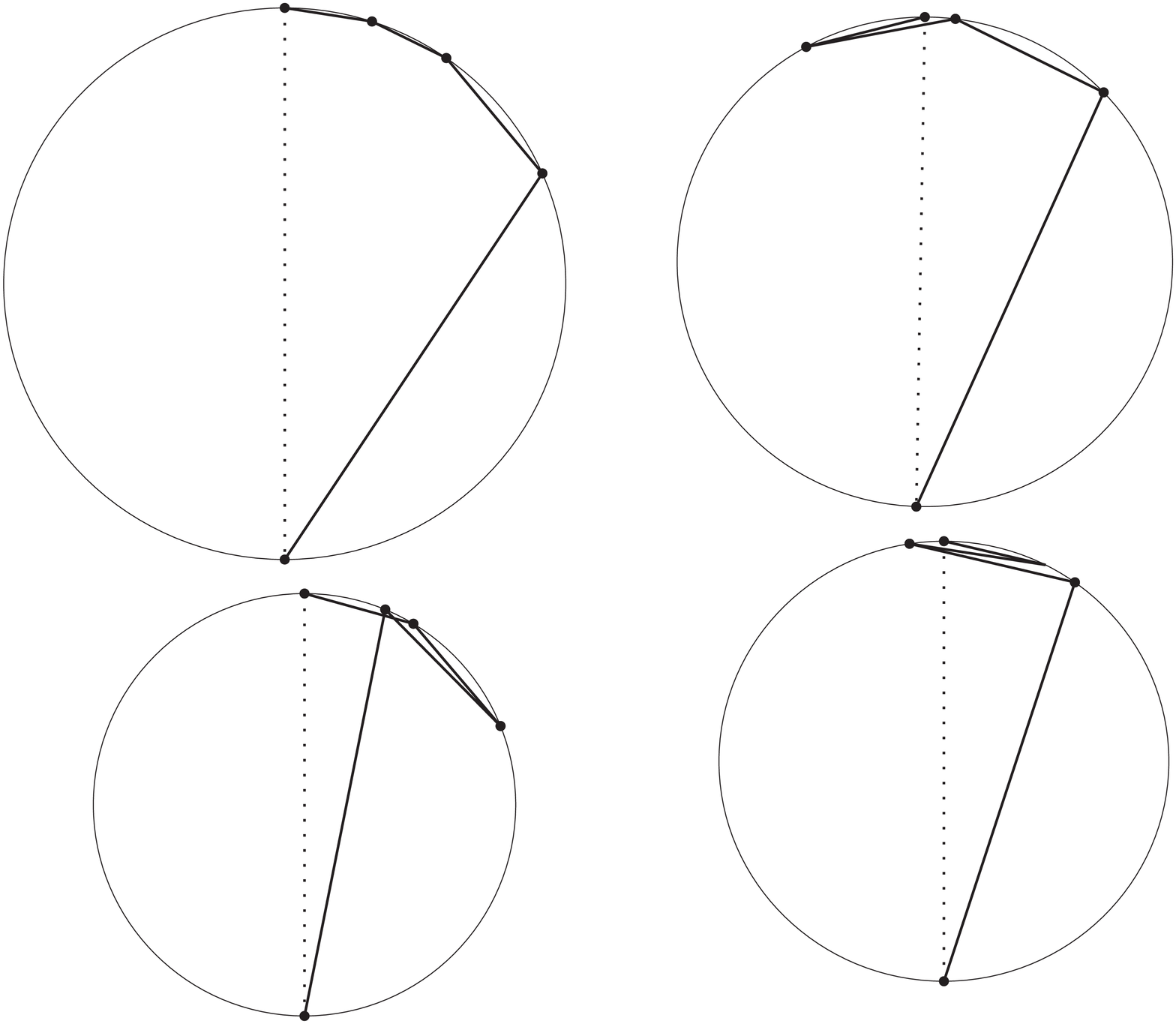}
\caption{A is a perfect Morse function for
$L=(10,3,2,1)$}\label{Fig_openarm}
\end{figure}

\begin{figure}[h]
\centering
\includegraphics[width=12 cm]{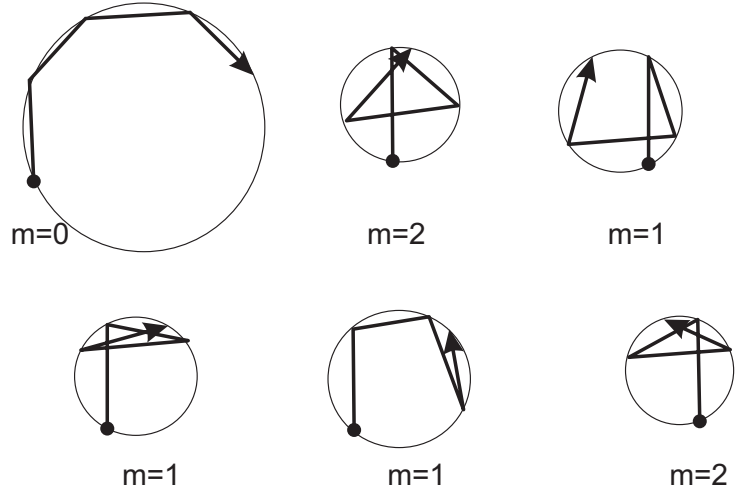}
\caption{A is not a perfect Morse function for
$L=(22,17,21.9,19)$}\label{Fig_openarmnonmorse}
\end{figure}

Now we are going to find the Morse index of a diacyclic
configuration of a robot arm by reducing the problem to the Morse
index of a critical configurations of some closed linkage. First, we
remind the reader the details about closed linkages. A
\textit{closed linkage} can be described as a flexible polygon on a
plane. It is defined by its string of edges $L = (l_1,\ldots,l_n), L
\in \mathbb{R}^n_+$. A \emph{configuration of a closed linkage} is
defined as a n-tuple of points $P=(p_1,\ldots,p_n)$ in the Euclidean
plane $\mathbb{R}^2$ such that $|p_{i}p_{i+1}|=l_i, i = 1, \ldots,
n$. Here the numeration is cyclic, i.e. $p_{n+1}=p_1$.
\begin{dfn}
For any configuration $P$ of $L$ with vertices $p_i = (x_i, y_i), i
= 1,\ldots, n$, its \emph{(doubled) oriented area } $A(R)$ is
defined as
$$2A(P) = (x_1y_2 - x_2y_1) + \cdots+(x_ny_1 - x_1y_n).$$
\end{dfn}

 Generically, the oriented area function is a Morse function on moduli space of a closed
linkage.

\begin{thm}\label{Thm_crirical_are_cyclic}(\cite{khipan})
 Generically, a polygon $P$ is a critical point of the
oriented area function $A$  iff $P$ is a cyclic configuration.
   \qed
\end{thm}

We will use the following notations for cyclic configurations, both
open and closed:

$O$ is the center of the circumscribed circle.

 $\alpha_i$  is the half of the angle between the vectors
$\overrightarrow{Op}_i$ and $\overrightarrow{Op}_{i+1}$. The angle
is defined to be positive, orientation is not involved.

Each edge  has an orientation $\varepsilon_i$ with respect to the
circumscribed circle:
 $$\varepsilon_i=\left\{
                       \begin{array}{ll}
                         1, & \hbox{if the center $O$ lies to the left of } p_ip_{i+1};\\
                         -1, & \hbox{if the center $O$ lies to the right of } p_ip_{i+1}.
                       \end{array}
                     \right.$$

$E(P)=(\varepsilon_1,\dots,\varepsilon_n)$ is the string of
orientations of all the edges.

$e(P)$ is the number of positive entries in $E(P)$.

$\mu_P=\mu_P(A)$ is the Morse index of the function $A$ at the point
$P$.

\begin{figure}\label{firstExample}
\centering
\includegraphics[width=8 cm]{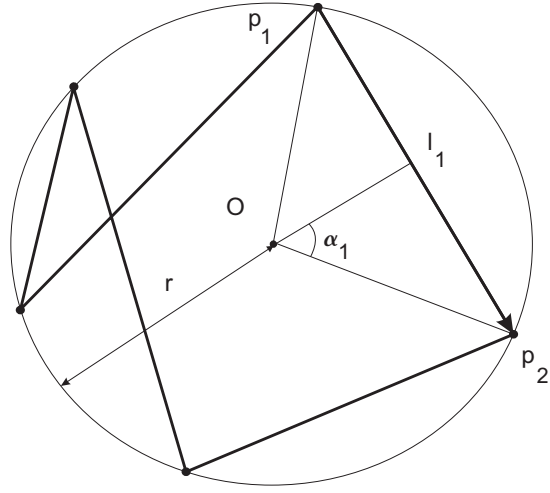}
\caption{Notation for a pentagonal cyclic configuration with
$E=(-1,-1,-1,1,-1)$}
\end{figure}

For cyclic configuration $P$ of a closed linkage $\omega_P$ is the
winding number of $P$ with respect to the center $O$.

\begin{thm}\label{Thm_Morse_closed_plane}(\cite{kpsz})
For a generic cyclic configuration $P$ of a closed linkage $L$,
$$\mu_P(A)=\left\{
       \begin{array}{ll}
         e(P)-1-2\omega_P &\hbox{if }\   \delta(P)>0; \\
         e(P)-2-2\omega_P & \hbox{otherwise}.
       \end{array}
     \right.$$
 $$\hbox{Here} \ \ \delta P=\sum_{i=1}^n \varepsilon_i \tan \alpha_i.\qed$$
\end{thm}

Returning to open chains,
 let $R$ be a
diacyclic configuration. Define its \textit{closure} $R^{Cl}$  as a
closed cyclic polygon obtained from $R$ by adding two positively
oriented edges (see Fig. \ref{Fig_openarm}) and denote by $\omega_R$
the winding number of the polygon $R^{Cl}$ with respect to the
center $O$. After this preparation we can present the below formula
for the Morse index.

\begin{thm}\label{open} Let $L=(l_1,\ldots,l_n)$ be a generic open
linkage, and let $R$ be one of its critical configuration.  For the
Morse index $\mu_R(A)$ of the oriented area function $A$ at the
point $R$, we have

$$\mu_R(A)=\left\{
         \begin{array}{ll}
           e(R)-2\omega_R+1 & \hbox{if ~} \delta(R)>0, \\
           e(R)-2\omega_R & \hbox{otherwise.}
         \end{array}
       \right.
$$
 $$\hbox{Here} \ \ \delta R=\sum_{i=1}^n \varepsilon_i \tan \alpha_i.$$

\end{thm}
\begin{proof}

\begin{figure}
\centering
\includegraphics[width=14 cm]{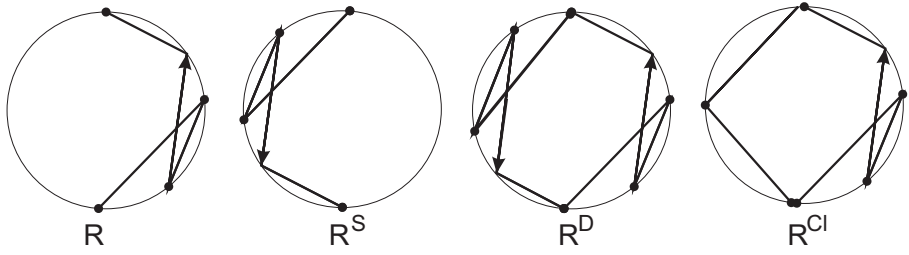}
\caption{An open chain, its symmetry image, duplication and closure
}\label{Fig_openarm}
\end{figure}

\begin{figure}
\centering
\includegraphics[width=12 cm]{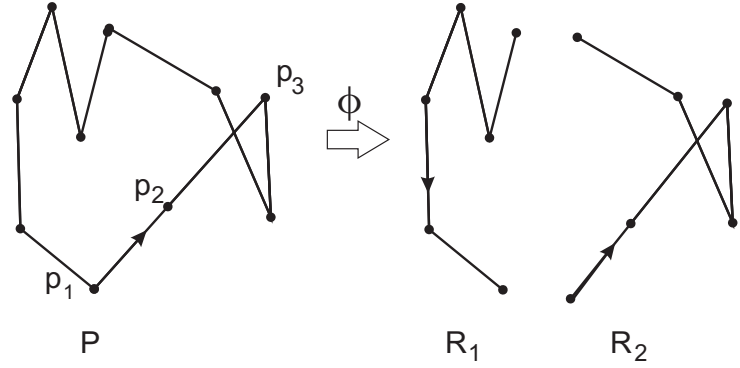}
\caption{The mapping $\phi$ splits a closed chain into two open
chains}\label{Fig_openarm_split}
\end{figure}

Consider the manifold $M^\circ_2(L) \times M^\circ_2(L)=\{R_1\times
R_2: \ R_1,R_2 \in M^\circ_2(L)\}$. Generically, the function
$A(R_1\times R_2)=A(R_1) +A(R_2)$ is a Morse function on
$M^\circ_2(L) \times M^\circ_2(L)$.

Next, define the \textit{duplication} of $L$ as the closed linkage
$L^D=(l_1,l_2,\ldots,l_n,l_1,l_2,\ldots,l_n)$.

Consider a mapping $\phi$ which splits a polygon $P \in L^D$ into
two open chains, $R_1$ and $R_2$.  The mapping  $\phi$ embeds
$M_2(L^D)$ as a codimension one submanifold of $M^\circ_2(L) \times
M^\circ_2(L)$.

For a cyclic open chain $R$, define $R^S$ as the symmetric image of
$R$ with respect to the center $O$. Define also $R^D\in M_2(L^D)$ as
a cyclic closed polygon obtained by patching together $R$ and $R^S$.
By Theorem \ref{Thm_crirical_are_cyclic}, $R^D$ is a critical point
of the oriented area.

On the one hand, the Morse index of its image $\phi(R^D) = R\times
R^S$ on the manifold $M^\circ_2(L) \times M^\circ_2(L)$ equals
$2\mu_R$. On the other hand, the Morse index of  $R^D$ on the
manifold $M_2(L^D)$ is known by Theorem
\ref{Thm_Morse_closed_plane}.

Since $M_2(L^D)$ embeds as a codimension one submanifold of
$M^\circ_2(L) \times M^\circ_2(L)$, the
 Morse indices  differ at most by one. More
precisely, we have the following lemma:
\begin{lemma}\label{lemma_open_morse_plane}
 Either
$\mu_{R^D}=2\mu_R$,   or $\mu_{R^D}=2\mu_R-1.\qed$ \end{lemma}

By Theorem \ref{Thm_Morse_closed_plane},
$$\mu_{R^D}=\left\{
                 \begin{array}{ll}
                   e(R^D)-2\omega(R^D)-1 & \hbox{ if } \delta(R^D)>0,\\
                   e(R^D)-2\omega(R^D)-2 & \hbox{ otherwise.}
                 \end{array}
               \right.
$$

Clearly, we have $e(R^D)=2e(R)$, \  $\delta(R^D)=2\delta(R)$, and
$\omega(R^D)=2\omega(R)-1$. This gives us

$$\mu_{R^D}=\left\{
                 \begin{array}{ll}
                   2e(R)-4\omega(R)+1 & \hbox{if } \delta(R)>0,\\
                   2e(R)-4\omega(R) & \hbox{otherwise.}
                 \end{array}
               \right.
$$
Assume that  $\delta(R)>0$. Then $\mu_{R^D}=
 2e(R)-4\omega(R)+1$ which is an odd number. The only possible choice in Lemma \ref{lemma_open_morse_plane}
 is $2\mu_R=
 2e(R)-4\omega(R)+2$.

 Analogously, if $\delta(R)<0$ we conclude that $2\mu_R=
 2e(R)-4\omega(R)$.
\end{proof}

\begin{ex}
Figure \ref{Fig_openarm}   depicts a number of diacyclic
configurations for which we obviously have $\delta(R)>0$. The Morse
indices are calculated easily. The  robot arm in question has four
more diacyclic configurations symmetric to the depicted ones. For
them, we easily have Morse indices $2, 2, 2,$ and $0$.

The robot arm in Figure  \ref{Fig_openarmnonmorse} presents more
diacyclic configurations with their Morse indices.
\end{ex}

\section{Concluding remarks}

We now wish to outline certain of the natural problems and
perspectives suggested by the above results.

1. The most intriguing problem is to find an analog of the
generalized Heron polynomial for n-arm, i.e., a univariate
polynomial such that its roots give the critical values of area on
the moduli space of an arm. Specifically, find out what is the
minimal algebraic degree of such a polynomial. Existence of such a
polynomial follows from the general results of algebraic geometry
using elimination theory but this does not give sufficient
information about its algebraic degree.

2. Consider all n-arms with fixed n. What is the exact estimate for
the number of diacyclic configurations of such an n-arm? An estimate
is provided by the degree of generalized Heron polynomial of the
duplicate 2n-gon but this estimate is far not exact and the problem
remains unsolved starting with n=4. An exact estimate could be
obtained as the algebraic degree of a generalized Heron polynomial
sought in the first problem.

3. As we have shown, the oriented area may or may not be a perfect
 Morse function on the configuration space of n-arm. For
which collection of the lengths $l_i$ is it perfect, i.e. has the
minimal possible number of nondegenerate critical points equal to
the sum of Betti numbers of moduli space? In other words, we seek
for a criterion of perfectness of oriented area in terms of the
lengths of the links. A related problem is to find out if the area
can be a function with the minimal possible number of critical
points given by the Lusternik-Schnirelmann category of the moduli
space.  As we have seen, this is the case for equilateral 3-arms.
Does the same hold for equilateral 4-arms?

4. An interesting issue is suggested by our description of
quasi-cyclic configurations. Namely, as we have seen, each component
of quasi-cyclic configurations contains special points of three
types: diacyclic, closed cyclic and critical points of the square of
the connecting side. Are there any relations between the points of
these three types?

5. Analogous problems may be considered for configurations of an arm
in three-dimensional space.

\end{document}